\documentclass{amsart}
\usepackage{tikz}
\usetikzlibrary{decorations.markings}

\usepackage{float}
\usepackage{url}
\usepackage{longtable}
\usepackage{xspace}
\usepackage[e]{esvect}
\usepackage{amsmath}
\usepackage{amssymb}
\usepackage{subcaption}
\usepackage{enumerate}
\usetikzlibrary{patterns}
\usepackage{verbatim}
\usetikzlibrary {positioning}
  \usetikzlibrary{arrows,topaths}
   \usetikzlibrary{calc,intersections}
\usepackage{graphicx}
\usetikzlibrary{positioning, fadings, backgrounds}
\newtheorem{theorem}{Theorem}[section]
\newtheorem{lemma}[theorem]{Lemma}
\newtheorem{proposition}[theorem]{Proposition}

\newtheorem{conjecture}{Conjecture}[section]

\usepackage{float} 
\usepackage{stix} 
\usepackage{tikz}
\usepackage{caption}
\usepackage{enumitem}
\usepackage{xhfill}

\newcommand{\Po}{{P}}
\newcommand{\maxrts}{\max_\mathcal{R}\text{TS}}
\newcommand{\minrst}{\min_\mathcal{R}\text{ST}}
\newcommand{\maxrsts}{\max_{\mathcal{R}^*}\text{TS}}
\newcommand{\minrsst}{\min_{\mathcal{R}^*}\text{ST}}
\title{Two-count interval representation of a permutation}
\author{Csaba Bir\'o}
\address{Department of Mathematics, University of Louisville, Louisville, Kentucky, U.S.A.}
\email{csaba.biro@louisville.edu}

\author{Andr\'{e} E. K\'{e}zdy}
\address{Department of Mathematics, University of Louisville, Louisville, Kentucky, U.S.A.}
\email{kezdy@louisville.edu}

\author{Jen\H{o} Lehel}
\address{Alfr\'ed R\'enyi Institute of Mathematics, Budapest, Hungary} 
\email{lehelj@renyi.hu}

\date{\today}

\begin{document}

\begin{abstract}The interval count problem, a classical question in the study of interval orders, was introduced by Ronald Graham in the 1980s.
This problem asks: given an interval order $P$, what is the minimum number of distinct interval lengths
required to construct an interval representation of $P$?  Interval orders that can be represented with just one
interval length are known as semiorders, and their characterizations are well known.
However, the characterization of interval orders that require at
most $k$ interval lengths — termed $k$-count interval orders—remains an open and challenging problem for $k\geq 2$.

Our investigation into $2$-count interval orders led us naturally to consider a related problem,
interval representations of permutations, which we introduce in this paper.
Specifically, we characterize permutations that have a $2$-count interval representation.
We prove that a permutation admits a $2$-count interval representation if and only if its
longest decreasing subsequences have length at most $2$.
For larger values of $k$, however, a similar characterization does not hold.  
There are permutations that do not permit a $3$-count interval representation
despite having decreasing subsequences of length at most $3$.  
Characterizing $k$-count permutations remains open for $k \geq 3$.

The $k$-count permutation representation problem appears to capture essential aspects
of the broader problem of characterizing $k$-count interval orders. 
To support this connection, we apply our findings on interval representations of permutations
to demonstrate that a height-$3$ interval order is $2$-count if and only if it has depth at most $2$,
where the depth of an interval order refers to the length of the longest nested
chain of intervals required in any interval representation of the order.
\end{abstract}

\maketitle

\section{Introduction} 
A finite poset $P=(X, \prec)$ is an interval order if it has an interval representation, a 
family of compact intervals $\{I_x : x \in X\}$ in $\mathbb{R}$ such that $x \prec y$ if and only
if $I_x$ is entirely to the left of $I_y$, for all $x, y \in X$.  An interval representation of $P$ using $k$ distinct interval lengths is called a {\it $k$-count representation} of $P$. 
The {\it depth} of an interval order is  the maximum number of intervals in a chain of nested intervals need to occur in every representation of the interval order.
 
A systematic study of $k$-count representations for interval graphs and interval orders was initiated by Fishburn 
\cite{FishburnBook} a half a century ago. 
Fishburn proved \cite[Section 9, Theorem 1]{FishburnBook} that for arbitrary $k$, there are depth-2 interval orders with no $k$-count interval representation. He found the smallest $2$-count interval order with no $2$-count representation that has $10$ elements  
\cite[Section 9, Example 2]{FishburnBook}. 

The characterization of interval orders having a $k$-count representation and the complexity status of the corresponding recognition problems are still open for $k\geq 2$  in general. There are results for $k=2$ with additional restrictions \cite{BoyadzhiyskaDissertation,BGT,FMOS,JLORS}.
Here we discuss another particular case for $k=2$. Our main result is  Theorem \ref{2count}, where we give a characterization of height-3 posets having a 2-count interval representation.  
 This theorem supports our conjecture stated in Section \ref{main} on the forbidden structure characterization of a
{$\mathbf{4}+\mathbf{1}$}-free 
interval order to be 2-count ({$\mathbf{4}+\mathbf{1}$} is a 5-element poset, the union of a chain on four elements and a distinct incomparable element).

The heart of the proof of the main theorem is an application of Theorem \ref{2countperm} on the $2$-count interval representations of permutations. Section \ref{prelim} introduces the $k$-count problems for permutations. The proof of Theorem \ref{2countperm},  a procedure that solves a corresponding linear system of inequalities, is given in Section \ref{mutations}.

By the classical theorem of Scott and Suppes \cite{ScottSuppes}, an interval order has 1-count interval representation if and only if it contains no {$\mathbf{3}+\mathbf{1}$} subposet; equivalently if it is a semiorder ({$\mathbf{3}+\mathbf{1}$} is a 4-element poset, the union of a chain on three elements and a distinct incomparable element). 

A necessary condition for the existence of a $2$-count representation is that
the interval order has depth at most two, that is, there is no nested chain of more than two intervals, which occurs in every representation of the interval order (see  Proposition \ref{nested}). Notice that forbidding a {$\mathbf{4}+\mathbf{1}$} subposet from a depth-2  interval order is still not sufficient to guarantee a $2$-count interval representation.
For convenience, we identify an interval order with its ascent sequence (see Bousquet-Mélou et al. \cite{ascent2010}).  So, to introduce the interval order $P$ corresponding to the ascent sequence $(0,1,2,0,3,2,3,0,2,4,2)$ we write
$P=(0,1,2,0,3,2,3,0,2,4,2)$. The Hasse diagram and canonical representation of this interval order are presented  in Figure \ref{Andre11}. 
Computer search showed that $P=(0,1,2,0,3,2,3,0,2,4,2)$ is minimal in the sense that no depth-2, {$\mathbf{4}+\mathbf{1}$}-free interval order exists with less than 11 elements that has no 2-count representation. 

\section{Preliminaries}
\label{prelim}

\subsection{Interval orders}
\label{subintord}
The magnitude $m$ of an interval order $P$ is defined as the minimum number of endpoints of an interval representation of $P$. Greenough \cite{Greenough} proved that the minimal endpoint representation is essentially unique. The minimal endpoint representation of $P$ with  endpoints $0,1,\ldots,m-1$ is called the {\it canonical representation} of $P$. Greenough \cite{Greenough} also proved that 
 in the canonical representation every point is the left endpoint and the right endpoint of some interval. In particular, in a canonical representation  the singleton intervals  $[0,0]$ and $[m-1,m-1]$ are present. 

\tikzstyle{A} = [circle, draw=black!, fill=black,minimum width=.7pt, inner sep=.7pt]
\tikzstyle{X} = [circle,draw=black!,minimum width=2pt, inner sep=2pt]
 \tikzstyle{Xp} = [circle,draw=black!,minimum width=1pt, inner sep=1pt]
  \tikzstyle{Xr} = [circle, fill=red!30,draw=black!,minimum width=2pt, inner sep=1pt]
    \tikzstyle{Xred} = [circle, fill=red!40,draw=black!,minimum width=2pt, inner sep=2pt]
    \tikzstyle{Xg} = [circle, fill=cyan!40,draw=black!,minimum width=2pt, inner sep=1pt] 
        \tikzstyle{Xgreen} = [circle, fill=green!60,draw=black!,minimum width=2pt, inner sep=2pt] 
        \tikzstyle{Xb} = [circle, fill=violet!30,draw=black!,minimum width=2pt, inner sep=1pt] 
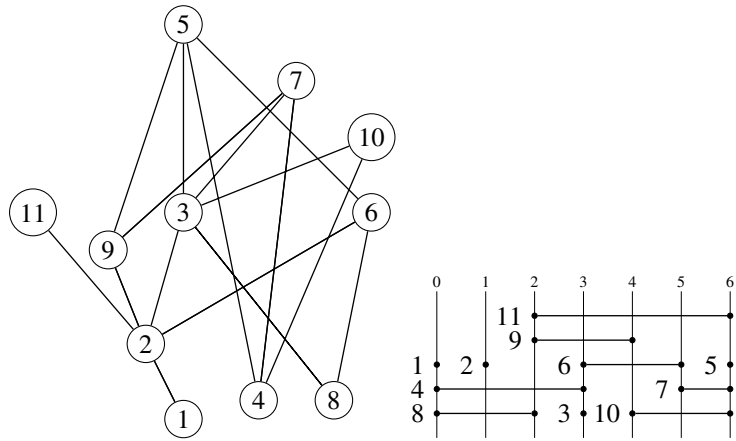
\begin{figure}[htp]   
    \begin{center}
\begin{tikzpicture}[scale=.5]

\node[X](1) at (3,-.5) {$1$};
\node[X](2) at (2,1.5) {$2$};
\node[X](8) at (7,0) {$8$};

\node[X](3) at (3,5) {$3$};
\node[X](6) at (8,5) {$6$};

\node[X](4) at (5,0) {$4$};

\node[X](5) at (3,10) {$5$};
\node[X](7) at (6,8.5) {$7$};
\node[X](10) at (8,7) {10};
\node[X](9) at (1,4) {$9$};
\node[X](11) at (-1,5) {$11$};

\draw[line width=.05em](1)--(2)--(9)--(5)--(4)--(7)--(3)--(10)--(4)(7)--(9);
\draw[line width=.05em](11)--(2)--(3)--(8)--(6)--(5)--(3)--(8)(2)--(6);

\draw[line width=.05em](4)--(7)--(9)--(2)--(1)(2)--(6);
\end{tikzpicture}
\begin{tikzpicture}[scale=.65]
\foreach \i in {0,...,6}{
    \draw[line width=.02em](\i,7)--(\i,10);
\node[]() at (\i,10.2) {\tiny\i};}
\node[A,label=left:$11$](L11) at (2,9.5) {};
\node[A](R11) at (6,9.5) {};
\draw[line width=.05em](L11)--(R11);

\node[A,label=left:$1$](LR1) at (0,8.5) {};
\node[A,label=left:$2$](LR2) at (1,8.5) {};
\node[A,label=left:$8$](L3) at (0,7.5) {};
\node[A](R3) at (2,7.5) {};
\draw[line width=.05em](L3)--(R3);
\node[A,label=left:$4$](L4) at (0,8) {};
\node[A](R4) at (3,8) {};
\draw[line width=.05em](L4)--(R4);
\node[A,label=left:$5$](LR5) at (6,8.5) {};

\node[A,label=left:$6$](L6) at (3,8.5) {};
\node[A](R6) at (5,8.5) {};
\draw[line width=.05em](L6)--(R6);

\node[A,label=left:$7$](L7) at (5,8) {};
\node[A](R7) at (6,8) {};
\draw[line width=.05em](L7)--(R7);
\node[A,label=left:$3$](LR3) at (3,7.5) {};

\node[A,label=left:$9$](L9) at (2,9) {};
\node[A](R9) at (4,9) {};
\draw[line width=.05em](L9)--(R9);
\node[A,label=left:$10$](L10) at (4,7.5) {};
\node[A](R10) at (6,7.5) {};
\draw[line width=.05em](L10)--(R10);
\end{tikzpicture}
\end{center}
\caption{Hasse diagram and canonical representation of the smallest {$\mathbf{4}+\mathbf{1}$}-free and depth-2 interval order   $P=(0,1,2,0,3,2,3,0,2,4,2)$  that has no  $2$-count representation}
\label{Andre11}
\end{figure} 
By the classical theorem of Scott and Suppes \cite{ScottSuppes}, an interval order has 1-count interval representation if and only if it contains no {\bf 3+1} subposet; equivalently if it is a semiorder.
 Each copy of a {\bf 3+1} subposet in $P$ determines a {\it peel} and  {\it pith} pair in the representation of $P$, that is, the singleton element $x$ and the middle element $y$ of the 3-chain, respectively (see Figure \ref{ppair}).  In our diagrams a peel colored  red is longer than a pith colored green.

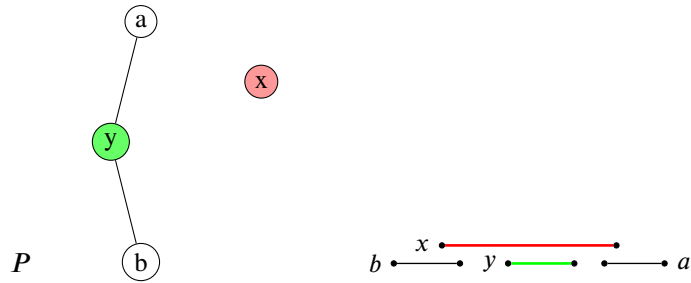
\begin{figure}[htp]
\begin{center}
\begin{tikzpicture}[scale=.8]
\node[Xgreen](y) at (0,2)  {y}; \node[]()at(-1.5,0){\large$P$};
\node[X](a) at (0.5,4){a};
\node[X](b) at (0.5,0){b};
\draw[](a)--(y)--(b);
\node[Xred](x) at (2.5,3)  {x};
\end{tikzpicture}
\hskip1cm
\begin{tikzpicture}[scale=.8]
\node[A,label=left:{$x$}](L1) at (.3,-0.3){};\node[A](R1) at (3.2,-0.3){};
\node[A,label=left:{$y$}](L2) at (1.4,-.60){};\node[A](R2) at (2.5,-.60){};
\node[A,label=left:{$b$}](L3) at (-.5,-.60){};\node[A](R3) at (.6,-.60){};
\node[A](L4) at (3,-.60){};\node[A,label=right:{$a$}](R4) at (4,-.60){};
\draw[line width=.1em,red](L1)--(R1);
\draw[line width=.1em,green](L2)--(R2);
\draw[line width=.05em](L3)--(R3);\draw[line width=.05em](L4)--(R4);
\end{tikzpicture}
\end{center}
\caption{Peel-pith pair in poset $P$ and its interval representation}
\label{ppair}
\end{figure}

Peel-pith pairs in an interval order $P=(X,\prec)$ can be presented in form of an auxiliary graph. Denote  PP the digraph on  vertex set $X$
and an arrow  from the pith to the peel occurring in a copy of  {\bf 3+1} subposet in $P$ , for each peel-pith pair.

\begin{figure}[htp]
\begin{center}
\begin{tikzpicture}[scale=.25]
\node[X]() at (0,0) {$1$};
\node[X]() at (6,0) {$10$};
\node[X]() at (0,6) {$7$};
\node[X]() at (6,6) {$5$};
\end{tikzpicture}
\hskip1cm
\begin{tikzpicture}[scale=.25]
\node[X](4) at (0,0) {$4$};
\node[X](2) at (3,6) {$2$};
\node[X](8) at (6,0) {$8$};
\draw[->,line width=.08em] (2)--(4);\draw[->,line width=.08em] (2)--(8);
\end{tikzpicture}
\hskip1cm
\begin{tikzpicture}[scale=.25]
\node[X](9) at (0,0) {$9$};
\node[X](1) at (6,0) {$11$};
\node[X](3) at (0,6) {$3$};
\node[X](6) at (6,6) {$6$};
\draw[->,line width=.08em] (3)--(9);\draw[->,line width=.08em] (3)--(1);
\draw[->,line width=.08em] (6)--(1);
\end{tikzpicture}
\end{center}
\caption{The PP graph of $P=(0,1,2,0,3,2,3,0,2,4,2)$}
\label{PPP}
\label{cpair}
\end{figure}
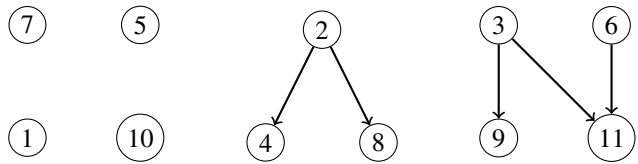

In the next proposition we state the obvious fact that the inclusion of peel-pith pairs  is forced in every representation of an interval order. 
\begin{proposition}
\label{nested} An antichain $A=\{a_1,\ldots,a_t\}$ of an interval order $P$ forms a nested chain of intervals in  every representation of $P$
if and only if  $A$ induces a directed path in the {\em PP}-graph of $P$.
\end{proposition}

 The number of elements of a longest directed path in graph PP is the {\it depth} of the interval order.
The {\it height} of an interval order is the number of elements in a longest chain of the poset.
A semiorder has depth one since it is $\mathbf{3}+\mathbf{1}$-free, thus its PP has no arc. On the other hand, semiorders can have arbitrary height since a chain is a semiorder. 

Here we investigate  2-count representations of an interval order $P$. We assume that $P$ has depth 2; this is an obvious necessary condition for the existence of a 2-count representation. 
A 2-count representation produces a $2$ coloring that is a partition $T\cup S$ of the elements of $P$, where
 $T$ is the set of  elements represented with long intervals, and $S$  is the set of  elements represented with short  intervals.
 In the diagrams long intervals and  elements of $T$ are colored red,
 short intervals and elements of $S$ are colored green. 
  A red/green coloring of an interval order $P$  is {\it feasible} if it comes from a 2-count representation of $P$. 

Besides peel-pith pairs, there are further constraints to classifying intervals to be short or long  in every representation. We introduce a family of 
 posets embodying such constraints. 
A {\it spring} pattern\footnote{\ The name comes from a physical metaphor identifying interval length constraints as acting forces.} is a partially $2$-colored poset with six elements defined by the Hasse diagram in  Figure \ref{Sp1}, where  $4,5\in T$, $2,3\in S$ and $1,6$ have arbitrary color. Notice that $5,3$ form a peel-pith pair. The family of springs consists of all partially colored posets that are defined by Figure \ref{Sp1}  and the dual of these posets. 

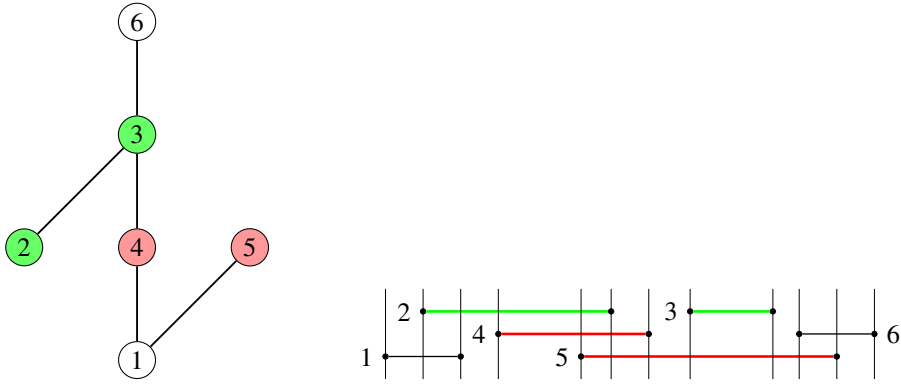
\begin{figure}[htp]
\begin{center}
\begin{tikzpicture}
\node[X](v) at (1,4) {6};
\node[Xgreen](d2) at (1,2.5) {3};
\node[Xred](tx) at (1,1) {4};
\node[X](u) at (1,-.5) {1};
\node[Xgreen](d1) at (-.5,1) {2};
\node[Xred](ty) at (2.5,1) {5};
\draw[line width=.07em] (d1)--(d2)(v)--(d2)--(tx)--(u)--(ty);
\end{tikzpicture}
\hskip1cm
\begin{tikzpicture}
\node[A,label=left:{$4$}](L0) at (-1,0.){};\node[A](R0) at (1,0.){};
\node[A,label=left:{$5$}](L1) at (.1,-0.3){};\node[A](R1) at (3.5,-0.3){};
\node[A,label=left:{$3$}](L2) at (1.55,0.3){};\node[A](R2) at (2.65,0.3){};
\node[A,label=left:{$2$}](L3) at (-2,0.3){};\node[A](R3) at (.5,0.3){};
\node[A](L4) at (3,0){};\node[A,label=right:{$6$}](R4) at (4,0){};
\node[A,label=left:{$1$}](L) at (-2.5,-.3){};\node[A](R) at (-1.5,-.3){};
\draw[line width=.1em,red](L0)--(R0);
\draw[line width=.1em,red](L1)--(R1);
\draw[line width=.1em,green](L2)--(R2);
\draw[line width=.1em,green](L3)--(R3);
\draw[line width=.05em](L4)--(R4);\draw[line width=.05em](L)--(R);
\foreach \i in {-2.5,-2,-1.5,-1,.1,.5,1,1.55,2.65,3,3.5,4}{
    \draw[line width=.02em](\i,-.6)--(\i,.6);
}
\end{tikzpicture}
\caption{A spring pattern and its interval representation that can not be made $2$-count}
\label{Sp1}
\end{center}
\end{figure}

\begin{proposition}\label{prop:nospring} A feasible 2-coloring  contains no spring. 
\end{proposition}
\begin{proof} Assume that  there is a $2$-count interval  representation of a spring, $[\ell_x,r_x]_{x=1}^6$, where 
the endpoints must satisfy
\begin{eqnarray}
\label{a} \ell_2\leq \ell_1<\ell_4,\ell_5\leq r_2,& \text{\ and}\\
\label{b} \ell_2 \leq  r_1<\ell_4\leq  r_4<\ell_3\leq r_3< \ell_6\leq r_5,& 
\end{eqnarray}
as seen in Figure \ref{Sp1}. By assumption, $2,3\in S$ and $4,5\in T$ in the feasible $2$-coloring. 
Therefore, 
\begin{eqnarray}
\label{short} r_2-\ell_2 &=&r_3-\ell_3 , \text{\ and}\\
\label{long} r_4-\ell_4 &=&r_5-\ell_5.
\end{eqnarray}
Using  (\ref{a}),  (\ref{short}) and   (\ref{b}) we obtain 
 $\ell_5-\ell_4 < r_2-\ell_2 = r_3-\ell_3 < r_5-r_4,$ 
contradicting  (\ref{long}).
\end{proof}

It is straightforward to check that the interval order $P=(0,1,2,0,3,2,3,0,2,4,2)$ in Figure \ref{Andre11} is {$\mathbf{4}+\mathbf{1}$}-free;
the PP graph in Figure \ref{PPP} shows that it  has  depth-2. To verify that $P$ has no 2-count representation notice that the subset 
$\{1,2,4,6,7,9\}$ 
induces in $P$ a spring, then apply Proposition \ref{prop:nospring}. 

 By the classical theorem of Scott and Suppes \cite{ScottSuppes}, an interval order has 1-count interval representation if and only if it contains no {$\mathbf{3}+\mathbf{1}$} subposet; equivalently if it is a semiorder. 
While the exclusion of {$\mathbf{3}+\mathbf{1}$} subposets  completely characterizes  1-count interval orders, forbidding  {$\mathbf{4}+\mathbf{1}$} subposets from interval orders is not consequential regarding the existence of a 2-count interval representation. However,  we conjecture that that the exclusion of {$\mathbf{4}+\mathbf{1}$} posets and the springs yields a particular family of interval orders having a 2-count interval representation.
\begin{conjecture} 
\label{4plus1}
A {$\mathbf{4}+\mathbf{1}$}-free, depth-2 interval order $\Po$ admits 
a 2-count representation if and only if $P$ has a $2$-coloring such that
every peel is red, every pith is green, and no spring pattern occurs as a colored subposet.
\end{conjecture}
 
In Section \ref{main} we prove a test case of Conjecture \ref{4plus1}, where the height
of $\Po$ is 3. In fact, in this case, 
neither a copy of {$\mathbf{4}+\mathbf{1}$} nor a spring appear as a subposet of $\Po$, since each has height 4. 

\subsection{Interval representation of permutations}
\label{subpermutation}
 An {\it interval representation}  of a permutation $\pi$ of $[n]$ is a family of real intervals $ [\ell_j,r_j]_{i=1}^n$ with integral endpoints satisfying
 \begin{equation}\label{endpoints}
\ell_{\pi(1)}<\cdots< \ell_{\pi(n)}<r_1< \cdots<r_n, 
\end{equation}
that is,  the right endpoints are in natural order, and the left endpoints correspond to the permutation  
   (see an example in Figure \ref{abra1}). 

 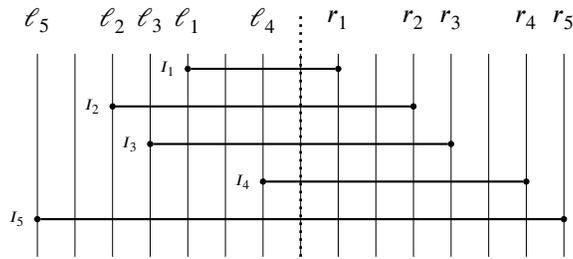
\begin{figure}[htp]
\begin{center}
\begin{tikzpicture}
\foreach \i in {-7,...,7}{
    \draw[line width=.01em](.5*\i,.5)--(.5*\i,3.2);}
  \node[]() at (-7*.5,3.65) {$\ell_5$}; 
  \node[]() at (-5*.5,3.65) {$\ell_2$};
  \node[]() at (-4*.5,3.65) { $\ell_3$};
  \node[]() at (-3*.5,3.65) { $\ell_1$};
  \node[]() at (-.5,3.65) { $\ell_4$};
  
   \node[]() at (7*.5,3.65) {$r_5$}; 
  \node[]() at (3*.5,3.65) {$r_2$};
  \node[]() at (4*.5,3.65) { $r_3$};
  \node[]() at (.5,3.65) { $r_1$};
  \node[]() at (3,3.65) { $r_4$};
\node[A,label=left:\tiny$I_5$](L4) at (-7*.5,1) {};
\node[A,label=right:\tiny$$](R4) at (7*.5,1) {};
\draw[line width=.08em](L4)--(R4);
\node[A,label=left:\tiny$I_2$](L2) at (-5*.5,2.5) {};
\node[A,label=right:\tiny$$](R2) at (3*.5,2.5) {};
\draw[line width=.08em](L2)--(R2);
\node[A,label=left:\tiny$I_3$](L3) at (-4*.5,2) {};
\node[A,label=right:\tiny$$](R3) at (4*.5,2) {};
\draw[line width=.08em](L3)--(R3);
\node[A,label=left:\tiny$I_1$](L5) at (-3*.5,3) {};
\node[A,label=right:\tiny$$](R5) at (.5,3) {};
\draw[line width=.08em](L5)--(R5);
\node[A,label=left:\tiny$I_4$](L1) at (-.5,1.5) {};
\node[A,label=right:\tiny$$](R1) at (6*.5,1.5) {};
\draw[line width=.08em](L1)--(R1);
 \draw[line width=.1em,dotted](0,.5)--(0,3.75);
\end{tikzpicture}
\end{center}
\caption{An interval representation of $\pi=[5,2,3,1,4]$}
\label{abra1}
\end{figure}

A permutation $\pi$ is {\it $k$-count} if  
$\pi$ has
an interval representation such that  the number of distinct interval lengths is $k$. 
The {\it depth} of $\pi$ is defined as the length of the longest decreasing subsequence of $\pi$, that is a longest chain of nested intervals in an  interval representation of $\pi$; if this number is $k$, then $\pi$ is called a {\it depth-k permutation}. 

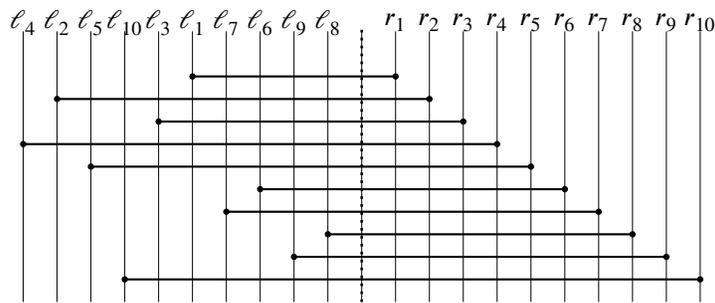
\begin{figure}[htp]
\begin{center}
\begin{tikzpicture}[scale=.6]
\foreach \i in {-10,...,10}{
   \draw[line width=.02em](.75*\i,-5)--(.75*\i,1);}
   \foreach \i in {1,...,10}{
   \node[]()at(.75*\i,1.2){$r_{\i}$};}
     \node[]()at(-.75*10,1.2){$\ell_4$};  \node[]()at(-.75*9,1.2){$\ell_2$};  \node[]()at(-.75*8,1.2){$\ell_5$};  \node[]()at(-.75*7,1.2){$\ell_{10}$};
       \node[]()at(-.75*6,1.2){$\ell_3$};  \node[]()at(-.75*5,1.2){$\ell_1$};  \node[]()at(-.75*4,1.2){$\ell_7$};  \node[]()at(-.75*3,1.2){$\ell_6$};
         \node[]()at(-.75*2,1.2){$\ell_9$};  \node[]()at(-.75*1,1.2){$\ell_8$};
 
\draw[line width=.08em] (-.75*5,-0.5*0) -- node [above] {$$} (0.75*1,-0.5*0); 
\draw[line width=.08em] (-.75*6,-0.5*2) -- node [above] {$$} (0.75*3,-0.5*2); 
\draw[line width=.08em] (-.75*9,-0.5*1) -- node [above] {$$} (0.75*2,-0.5*1); 
\draw[line width=.08em] (-.75*10,-0.5*3) -- node [above] {$$} (0.75*4,-0.5*3); 
\draw[line width=.08em] (-.75*8,-0.5*4) -- node [above] {$$} (0.75*5,-0.5*4); 
\draw[line width=.08em] (-.75*4,-0.5*6) -- node [above] {$$} (0.75*7,-0.5*6); 
\draw[line width=.08em] (-.75*3,-0.5*5) -- node [above] {$$} (0.75*6,-0.5*5); 
\draw[line width=.08em] (-.75*1,-0.5*7) -- node [above] {$$} (0.75*8,-0.5*7);
\draw[line width=.08em] (-.75*2,-0.5*8) -- node [above] {$$} (0.75*9,-0.5*8); 
\draw[line width=.08em] (-.75*7,-0.5*9) -- node [above] {$$} (0.75*10,-0.5*9);
   \draw[line width=.1em,dotted](0,-5)--(0,1.1);
\node[A]()at  (-.75*5,-0.5*0) {};\node[A]()at (0.75*1,-0.5*0){};
\node[A]()at (-.75*6,-0.5*2) {};\node[A]()at(0.75*3,-0.5*2){};
\node[A]()at (-.75*9,-0.5*1){};\node[A]()at (0.75*2,-0.5*1){};
\node[A]()at  (-.75*10,-0.5*3) {};\node[A]()at(0.75*4,-0.5*3){};
\node[A]()at  (-.75*8,-0.5*4) {};\node[A]()at (0.75*5,-0.5*4){};
\node[A]()at  (-.75*4,-0.5*6){};\node[A]()at (0.75*7,-0.5*6){};
\node[A]()at (-.75*3,-0.5*5) {};\node[A]()at(0.75*6,-0.5*5){};
\node[A]()at  (-.75*1,-0.5*7) {};\node[A]()at (0.75*8,-0.5*7){}; 
\node[A]()at (-.75*2,-0.5*8) {};\node[A]()at(0.75*9,-0.5*8){};
\node[A]()at  (-.75*7,-0.5*9) {};\node[A]()at (0.75*10,-0.5*9){};
\end{tikzpicture}
\end{center}
\caption{ $\pi=[4,2,5,10,3,1,7,6,9,8]$, a smallest depth-3 permutation with no $3$-count representation}
\label{abra6}
\end{figure}

The depth and interval count numbers have the obvious relation that any interval representation of a depth-$k$ permutation uses at least $k$ distinct  lengths in a nested chain of intervals. 
 We prove in Section \ref{mutations} that  
{a permutation  has a 2-count  interval representation if and only if it has depth at most $2$}
 (Theorem \ref{2countperm}).  
 
 Note that equality between depth and interval count numbers similar to Theorem \ref{2countperm} does not hold for depth-$k$ permutations, $k\geq 3$. 
There are several examples of depth-3 permutations admitting no
3-count integral representation.
 Computer search shows that $\pi=[4,2,5,10,3,1,7,6,9,8]$ is a  smallest depth-3 permutation with no 3-count representation. An interval representation of this permutation is displayed in Figure \ref{abra6}. 
 This permutation has no $3$-count representation. We include a verification 
at the end of the paper.
 
An indexed partition $(T_1,\ldots,T_k)$ of $[n]$ is called a \emph{sorted $k$-coloring}, 
if for all inversions $(x,y)$ such that $x\in T_i$ and $y\in T_j$, we have $i<j$. In terms of an interval representation $\{I_j\}_{j\in[n]}$ of a permutation, $I_x\subsetneq I_y$ implies that  $x\in T_i$ and  $y\in T_j$ with $i<j$. 

\begin{lemma}
\label{kcolor}
If the depth of a permutation $\pi$ is $k$, then $\pi$ has a sorted $k$-coloring.
\end{lemma}

\begin{proof}
Consider the inclusion order $P$ of the intervals in any interval representation  of $\pi$.   Notice that the chains of $P$ correspond to the decreasing subsequences of $\pi$, so the height of $P$ is $k$. By successive removal of minimal elements, $P$ can be partitioned into antichains $T_1,\ldots,T_k$ 
(c.f.\ Mirsky's Theorem \cite{Mirsky}). 
It is immediate from the definition of $P$ that these form a sorted $k$-coloring.
\end{proof}

A sorted $k$-coloring $T_1,\ldots,T_k$ of $\pi$ is {\it feasible} if there are integers $\alpha_1<\cdots <\alpha_k$ and an interval representation $\{I_j\}_{j\in[n]}$  of $\pi$ such that  $|I_x|=\alpha_i$ for every $x\in T_i$.

\section{The $2$-count representation of depth-2 permutations}
\label{mutations}
In this section permutations having a 2-count  interval representation are characterized. 
In the next theorem the bulk of the proof  is a procedure that verifies that
 every sorted 2-coloring of a depth-2 permutation is feasible.  
\begin{theorem}
\label{2countperm}
A permutation has a 2-count  interval representation if and only if it has depth at most 2.
\end{theorem}
\begin{proof} Intervals in a nested 3-chain have distinct lengths; therefore, having depth 2 (or less)  is necessary for the existence of a 2-count representation. 

Suppose that $\pi$ has depth 2, and let $\{I_j\}_{j=1}^n$, be an  interval representation of $\pi$. We apply Lemma \ref{kcolor} with $k=2$. Let  $[S,T]$ be a sorted $2$-coloring, a 2-partition of $[n]$ such that there are inversions neither in $S$ nor in $T$. Furthermore, $I_y\not\subset I_x$ for every $x\in S$, $y\in T$. 

 To see that $[S,T]$ is feasible, we need to find a 2-count representation  
 $\{[\ell_j,r_j]\}_{j=1}^n$, where the common  interval length in $S$  is $\alpha$, and the common interval  length in $T$ is $\beta>\alpha$. In other words, for a fixed permutation $\pi$ we need to find a feasible solution
  for the integral linear system subject to the conditions:
  \vskip.5em
 \begin{itemize}
\item[(a)] $\ell_{\pi(1)} < \cdots < \ell_{\pi(n)} < r_1<\cdots<r_n$,
\item[(b)] $\alpha < \beta$,
\item[(c)] $r_j-\ell_j=\alpha$\quad if $I_j\in S$,
\item[(d)] $r_j-\ell_j=\beta$\quad  if $I_j\in T$.
\end{itemize}
\vskip.5em
For an arbitrary integral interval representation $\mathcal R$ of $\pi$, (a) is equivalent with (\ref{endpoints}). 
If (b), (c), and (d) are satisfied with some $\alpha<\beta$,
then we have 
 \begin{eqnarray}\label{adjustleft}
 \ell_x\ =\ \  \left\{\begin{array}{ccc}
r_x-\alpha   & \text {if} & x\in S\\
\\
r_x-\beta  &  \text {if} & x\in T \ .
 \end{array}
  \right.
\end{eqnarray} 

Assume now that the right endpoints are fixed,  $r_1<\cdots<r_n$, and  the left endpoints are computed according to  (\ref{adjustleft}) with given $\alpha<\beta$. Then 
 the sequence of inequalities $\ell_{\pi(1)}<\cdots<\ell_{\pi(i)}<\ell_{\pi(i+1)}<\cdots <\ell_{\pi(n)}$ are satisfied if and only if $\ell_{\pi(i)}<\ell_{\pi(i+1)}$ holds for every $i$, $1\leq i<n$,  such that  $I_{\pi(i)}$ and $I_{\pi(i+1)}$ have distinct colors.
Let $x=\pi(i)$, $y=\pi(i+1)$, and assume that $I_x$ and $I_y$ have different colors.

 If  $x\in T$, $y\in S$ 
 then 
 $r_x-\beta =\ell_x < \ell_y=r_y-\alpha$ is satisfied if and only if
 \begin{equation}
 \label{lower}
r_x-r_y<\beta-\alpha .
\end{equation}

 The interval $[r_y,r_x]$ is called the {\it slack interval} of the pair $I_x, I_y$ $x\in T, y\in S$, and its length, $r_x- r_y$ defines  a {\it TS-slack}.   For fixed $r_1,\ldots,r_n$, each TS-slack is a lower bound on $\beta-\alpha$ as stated by (\ref{lower}).
  
 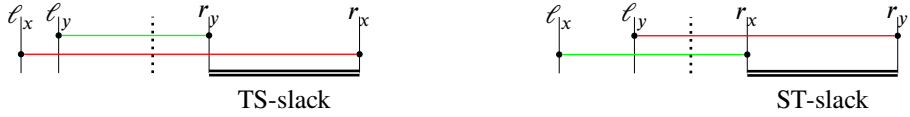
\begin{figure}[htp]
\begin{center}
\begin{tikzpicture}[scale=.5]
\foreach \i in {-2,-1,3,7}{
    \draw[line width=.02em](\i,1.5)--(\i,3);}
 
\node[A,label=left:\tiny$$](L1) at (-1,2.5) {};
\node[A,label=left:$$](R1) at (3,2.5) {};
\draw[line width=.05em,green](L1)--(R1);
\node[]() at (-1,3) {$\ell_y$};
\node[]() at (3,3) {$r_y$};

\node[A,label=left:\tiny$$](L2) at (-2,2) {};
\node[A,label=left:$$](R2) at (7,2) {};
\draw[line width=.05em,red](L2)--(R2);
\node[]() at (-2,3) {$\ell_x$};
\node[]() at (7,3) {$r_x$};
  \draw[line width=.1em,double](7,1.5)--(3,1.5) ;
\node() at (5,.8) {TS-slack};
       \draw[line width=.1em,dotted](1.5,1.5)--(1.5,3);
\end{tikzpicture}
\hskip2cm
\begin{tikzpicture}[scale=.5]
\foreach \i in {-1,1,4,8}{
    \draw[line width=.02em](\i,2)--(\i,3.5);}

\node[A,label=left:\tiny$$](L1) at (1,3) {};
\node[A,label=left:$$](R1) at (8,3) {};
\draw[line width=.05em,red](L1)--(R1);
\node[]() at (1,3.5) {$\ell_y$};
\node[]() at (8,3.5) {$r_y$};

\node[A,label=left:\tiny$$](L2) at (-1,2.5) {};
\node[A,label=left:$$](R2) at (4,2.5) {};
\draw[line width=.05em,green](L2)--(R2);
\node[]() at (4,3.5) {$r_x$};
\node[]() at (-1,3.5) {$\ell_x$};

  \draw[line width=.1em,double](8,2)--(4,2) ;
\node() at (6,1.3) {ST-slack};
            \draw[line width=.1em,dotted](2.5,2)--(2.5,3.5);
\end{tikzpicture}
\end{center}
\caption{Switching colors between intervals with consecutive left endpoints}
\label{slacks}
\end{figure}
If $x\in S$, $y\in T$ (with $x=\pi(i), y=\pi(i+1)$), then $I_y\not\subseteq I_x$ implies $r_x<r_y$. The interval  $[r_x,r_y]$ is called the slack interval of the pair $I_x, I_y$, $x\in S, y\in T$, and its length,  $r_y-r_x$, defines an {\it ST-stack}.
Since  $\ell_x<\ell_y$  if and only if  $r_x-\alpha<r_y-\beta$, equivalently:
\begin{equation}
\label{upper}
\beta-\alpha<r_y-r_x .
   \end{equation}
   For fixed $r_1,\ldots,r_n$, each ST-slack is an upper bound on $\beta-\alpha$ as stated by (\ref{upper}).
  
  \begin{lemma}\label{slackends}
    The left endpoint of a slack interval is the right endpoint of a green  interval, and the right endpoint of a slack interval is the right endpoint of a red  interval. 
    \end{lemma}
    \begin{proof}Figure \ref{slacks} shows that the claim is true in both cases, by definition of slack.   
    \end{proof}

   For an integral interval representation ${\mathcal{R}}$ of a permutation $\pi$ denote by $\maxrts$ the maximum of the TS-slack values and let $\minrst$ be the minimum of the ST-slack values. 
If   $\mathcal{R}$ is a  2-count interval representation,  then the inequalities (a) -- (d) imply
\begin{equation} \label{key}
\maxrts<\minrst,
\end{equation} 
We state and verify the converse in the next proposition.

\begin{proposition} 
\label{observ}
Let $\mathcal{R}$ be an integral interval representation of $\pi$, 
and let $[S,T]$ be a sorted $2$-coloring of $\pi$. 
If  \begin{equation*} 
\maxrts<\minrst,
\end{equation*} 
then a feasible solution of the system (a) - (d) is obtained by choosing a value $\beta - \alpha$ such that $\maxrts< \beta-\alpha <\minrst$;  furthermore,  by selecting a value $\alpha > r_{\pi(n)}-r_{1}$
if  $\pi(n)\in S$  or a value $\beta> r_{\pi(n)}-r_{1}$ if $\pi(n)\in T$. 
\end{proposition}
\begin{proof} To obtain a feasible  solution with lengths $\alpha$ and $\beta$ we modify $\mathcal R$ by keeping the right endpoints unchanged and by redefining the left endpoints using (\ref{adjustleft}). Let $\mathcal {R}^\prime$ = $[\ell^\prime_j,r^\prime_j]$, $j=1,\ldots,n$,
where $\ell^\prime_x=r_x- \alpha $ \ or $\ell^\prime_x=r_x-\beta$ \ and $r^\prime_x=r_x$. 
 We verify that  $\ell_{\pi(1)} < \cdots < \ell_{\pi(n)} $
 does not change in $\mathcal R^\prime$, that is $\ell^\prime_{\pi(i)}<\ell^\prime_{\pi(i+1)}$ if and only if $\ell_{\pi(i)}<\ell_{\pi(i+1)}$, for $i=1,\ldots,n-1$.
 
 Let $x=\pi(i), y=\pi(i+1)$, for some $1\leq i<n$. If $x,y\in S$ or if $x,y\in T$, then
 $\ell^\prime_x=r_x-\delta$ and $\ell^\prime_y=r_y-\delta$ with $\delta\in\{\alpha,\beta\}$;
 therefore, $\ell^\prime_x<\ell^\prime_y$ if and only if $r_x<r_y$.

Suppose  that $x\in T, y\in S$. Then $\ell^\prime_x<\ell^\prime_y$ if and only if $r_x-r_y<\beta -\alpha$. This last condition follows from the TS-slack inequality if  $r_x> r_y$, 
otherwise, it follows, since $r_x-r_y<0< \beta-\alpha$. 

Suppose that $x\in S, y\in T$. Then $\ell^\prime_x<\ell^\prime_y$ if and only if 
$\beta -\alpha<r_y-r_x$. Observe first that $r_y> r_x$, since $[S,T]$ is a sorted coloring of $\pi$
implying that $I_y\not\subset I_x$. Therefore, $\beta -\alpha<r_y-r_x$ holds due to the ST-slack inequality.

Thus the first and second half of (a), the inequalities  $\ell^\prime_{\pi(1)} < \cdots < \ell^\prime_{\pi(n)}$ and $r_1 < \cdots < r_n$ are satisfied. 
Finally, $\ell^\prime_{\pi(n)}<r_1$ follows from the lower bound set on either $\alpha$ or $\beta$.
 \end{proof}
 Note that $\alpha, \beta$ can be selected to be rational numbers, thus applying appropriate scaling, one obtains integral feasible solutions by Proposition \ref{observ}.\\
 
We proved in Proposition \ref{observ} that the key inequality (\ref{key}) leads to a feasible solution for the system (a) - (d).
Assume now that  
   (\ref{key}) is not satisfied by a integral interval representation of $\pi$.
Our strategy in this second part of the proof of the theorem 
consists of changing the position of the right endpoints  (by maintaining their order), and reach a representation 
  $\mathcal{R}^*$ of $\pi$ that satisfies  
  \begin{equation} \label{keyprime}
\maxrsts<\minrsst.
\end{equation}
Then a feasible solution is obtained by applying Proposition \ref{observ} with  ${\mathcal {R}}^*$ in the role of  ${\mathcal {R}}$. 
Therefore, our goal is to make all TS-slack values smaller than all ST-slack values.

\begin{lemma}
\label{lemma}
Let $J_1, J_2$ be distinct slack intervals in $\mathcal{R}$. 
\begin{itemize}
\item[(i)] If $J_1\subset  J_2$, then $J_1$ is a TS-slack interval, $J_2$ is an ST-slack interval, furthermore, they have a common left or a common right endpoint.
\item[ (ii)]  If $J_1\not\subset  J_2$, then $J_1$ and $J_2$ do not share a common endpoint.
\end{itemize}
\end{lemma}
\begin{proof} (i) 
 Let $J_1=[r_{x_1},r_{y_1}]$,  $J_2=[r_{x_2},r_{y_2}]$ be ST-slack intervals, where $x_1,x_2\in S$ and $y_1,y_2\in T$. Without loss of generality, let $\ell_{x_1}<\ell_{y_1}<\ell_{x_2}<\ell_{y_2}$ (see Figure \ref{abra3}). 

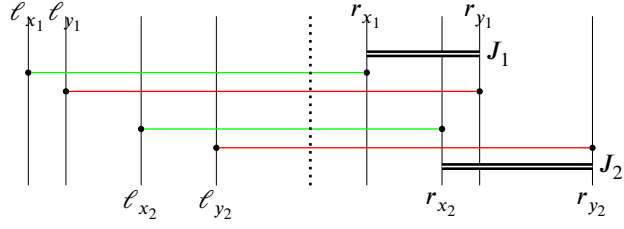
\begin{figure}[htp]
\begin{center}
\begin{tikzpicture}[scale=.5]
\foreach \i in {-5,-4,-2,0,4,6,7,10}{
    \draw[line width=.02em](\i,1.5)--(\i,6);}

\node[A,label=left:\tiny$$](L1) at (-4,4) {};
\node[A,label=left:$$](R1) at (7,4) {};
\draw[line width=.05em,red](L1)--(R1);
\node[]() at (-4,6) {$\ell_{y_1}$};
\node[]() at (7,6) {$r_{y_1}$};

\node[A,label=left:\tiny$$](L2) at (-5,4.5) {};
\node[A,label=left:$$](R2) at (4,4.5) {};
\draw[line width=.05em,green](L2)--(R2);
\node[]() at (-5,6) {$\ell_{x_1}$};
\node[]() at (4,6) {$r_{x_1}$};

  \draw[line width=.1em,double](4,5)--(7,5) ;
\node() at (7.5,5) {$J_1$};

\node[A,label=left:\tiny$$](L1) at (0,2.5) {};
\node[A,label=left:$$](R1) at (10,2.5) {};
\draw[line width=.05em,red](L1)--(R1);
\node[]() at (0,1) {$\ell_{y_2}$};
\node[]() at (10,1) {$r_{y_2}$};

\node[A](L2) at (-2,3) {};
\node[A,label=left:$$](R2) at (6,3) {};
\draw[line width=.05em,green](L2)--(R2);
\node[]() at (-2,1) {$\ell_{x_2}$};
\node[]() at (6,1) {$r_{x_2}$};
  \draw[line width=.1em,double](6,2)--(10,2) ;
\node() at (10.5,2) {$J_2$};
            \draw[line width=.1em,dotted](2.5,1.5)--(2.5,6);
\end{tikzpicture}
\end{center}
\caption{$J_1, J_2$ are not both ST-slack intervals}
\label{abra3}
\end{figure}
Since $I_{x_1}\not\subset  I_{x_2}$, we have $r_{x_1}<r_{x_2} $  that implies $r_{x_1}\in J_1\setminus J_2$ contradicting
$J_1\subset J_2$. Therefore, 
$J_1, J_2$ are not both ST-slack intervals.\\

Let  $J_2=[r_{y_2},r_{x_2}]$ be a TS-slack interval, where $x_2\in T, y_2\in S$. 
The condition  $J_1\subset J_2$  implies $r_{y_2}<r_{y_1}, r_{x_1}<r_{x_2}$
 (see Figure \ref{abra4} for the case $x_1\in T, y_1\in S$).
 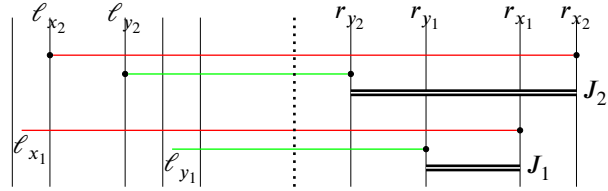
\begin{figure}[htp]
\begin{center}
\begin{tikzpicture}[scale=.5]
\foreach \i in{-5,-4,-2,-1,0,4,6,8.5,10}{
    \draw[line width=.02em](\i,1.5)--(\i,6);}
 
\node[A,label=left:\tiny$$](L1) at (-4,5) {};
\node[A,label=left:$$](R1) at (10,5) {};
\draw[line width=.05em,red](L1)--(R1);
\node[]() at (-4,6) {$\ell_{x_2}$};
\node[]() at (4,6) {$r_{y_2}$};

\node[A,label=left:\tiny$$](L2) at (-2,4.5) {};
\node[A,label=left:$$](R2) at (4,4.5) {};
\draw[line width=.05em,green](L2)--(R2);
\node[]() at (-2,6) {$\ell_{y_2}$};
\node[]() at (10,6) {$r_{x_2}$};

  \draw[line width=.1em,double](10,4)--(4,4) ;
\node() at (10.5,4) {$J_2$};

\node[label=left:\tiny$$](L1) at (-5,3) {};
\node[A](R1) at (8.5,3) {};
\draw[line width=.05em,red](L1)--(R1);
\node[]() at (-4.5,2.5) {$\ell_{x_1}$};
\node[]() at (6,6) {$r_{y_1}$};

\node[](L2) at (-1,2.5) {};
\node[A](R2) at (6,2.5) {};
\draw[line width=.05em,green](L2)--(R2);
\node[]() at (-0.5,2) {$\ell_{y_1}$};
\node[]() at (8.5,6) {$r_{x_1}$};

 \draw[line width=.1em,double](6,2)--(8.5,2) ;
\node() at (9,2) {$J_1$};
           \draw[line width=.1em,dotted](2.5,1.5)--(2.5,6);
\end{tikzpicture}
\end{center}
\caption{$J_2$  is not a TS-slack interval}
\label{abra4}
\end{figure}

Because $I_{x_1}\not\subset  I_{x_2}$ and $I_{y_2}\not\subset  I_{y_1}$, we have $\ell_{x_1}<\ell_{x_2}<\ell_{y_2}<\ell_{y_1}$. 
Then  $\{\ell_{x_1},\ell_{y_1}\}=\{\pi(i),\pi(i+1)\}$ is not true for any $1\leq i\leq n$, a contradiction. We conclude that $J_2$ is an ST-slack interval. As we saw earlier, one of $J_1$ and $J_2$ is a TS-slack, therefore, $J_1$ is a TS-slack interval. Let $J_1=[r_{y_1},r_{x_1}]$ and  $J_2=[r_{x_2},r_{y_2}]$. We claim that either $x_1=y_2$ or $y_1=x_2$.
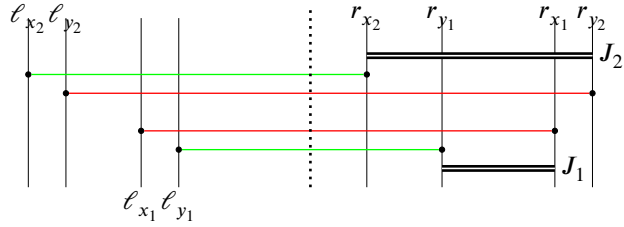
\begin{figure}[htp]
\begin{center}
\begin{tikzpicture}[scale=.5]
\foreach \i in {-5,-4,-2,-1,4,6,9,10}{
    \draw[line width=.02em](\i,1.5)--(\i,6);}
    
\node[A,label=left:\tiny$$](L1) at (-4,4) {};
\node[A](R1) at (10,4) {};
\draw[line width=.05em,red](L1)--(R1);
\node[]() at (-4,6) {$\ell_{y_2}$};
\node[]() at (10,6) {$r_{y_2}$};

\node[A,label=left:\tiny$$](L2) at (-5,4.5) {};
\node[A,label=left:$$](R2) at (4,4.5) {};
\draw[line width=.05em,green](L2)--(R2);
\node[]() at (-5,6) {$\ell_{x_2}$};
\node[]() at (4,6) {$r_{x_2}$};

 \draw[line width=.1em,double](10,5)--(4,5) ;
\node() at (10.5,5) {$J_2$};

\node[A,label=left:\tiny$$](L1) at (-2,3) {};
\node[A](R1) at (9,3) {};
\draw[line width=.05em,red](L1)--(R1);
\node[]() at (-2,1) {$\ell_{x_1}$};
\node[]() at (9,6) {$r_{x_1}$};

\node[A](L2) at (-1,2.5) {};
\node[A](R2) at (6,2.5) {};
\draw[line width=.05em,green](L2)--(R2);
\node[]() at (-1,1) {$\ell_{y_1}$};
\node[]() at (6,6) {$r_{y_1}$};

  \draw[line width=.1em,double](6,2)--(9,2) ;
\node() at (9.5,2) {$J_1$};
        \draw[line width=.1em,dotted](2.5,1.5)--(2.5,6.25);
\end{tikzpicture}
\end{center}
\caption{$I_{x_1}\neq I_{y_2}$ is impossible}
\label{abra5}
\end{figure}

Figure \ref{abra5} shows the case where $\ell_{x_2}<\ell_{y_2}\leq\ell_{x_1}<\ell_{y_1}$. If 
$I_{x_1}\neq I_{y_2}$, then $I_{x_1}\subset I_{y_2}$, a contradiction. Therefore,
$I_{x_1}= I_{y_2}$ and thus 
the right endpoints of $J_1$ and $J_2$ are identical. In the case of the order
$\ell_{x_1}<\ell_{y_1}\leq\ell_{x_2}<\ell_{y_2}$ we obtain $I_{y_1}=I_{x_2}$ similar way 
and 
the  left endpoints of $J_1$ and $J_2$ coincide (see Figure \ref{abra}). 
\begin{figure}[htp]
\begin{center}
\begin{tikzpicture}[scale=.5]
\foreach \i in {-4,-2,-1,4,6,9}{
    \draw[line width=.02em](\i,1.5)--(\i,6);}
    
\node[A,label=left:\tiny$$](L1) at (-1,3) {};
\node[A](R1) at (9,3) {};

\draw[line width=.05em,red](L1)--(R1);
\node[]() at (-4,6) {$\ell_{x_1}$};
\node[]() at (-2,6) {$\ell_{y_1}$};

 \draw[line width=.1em,double](4,5)--(6,5) ;
\node() at (6.5,5) {$J_1$};

\node[A,label=left:\tiny$$](L1) at (-2,3.5) {};
\node[A](R1) at (4,3.5) {};

\draw[line width=.05em,green](L1)--(R1);
\node[]() at (-2,1) {$\ell_{x_2}$};
\node[]() at (9,1) {$r_{y_2}$};

\node[A](L2) at (-4,4) {};
\node[A](R2) at (6,4) {};

\draw[line width=.05em,red](L2)--(R2);
\node[]() at (-1,1) {$\ell_{y_2}$};
\node[]() at (6,6) {$r_{x_1}$};
\node[]() at (4,1) {$r_{x_2}$};
\node[]() at (4,6) {$r_{y_1}$};
  \draw[line width=.1em,double](4,2)--(9,2) ;
\node() at (9.5,2) {$J_2$};
        \draw[line width=.1em,dotted](2.5,1.5)--(2.5,6.25);
\end{tikzpicture}
\end{center}
\caption{$I_{x_2}=I_{y_1}$}
\label{abra}
\end{figure}
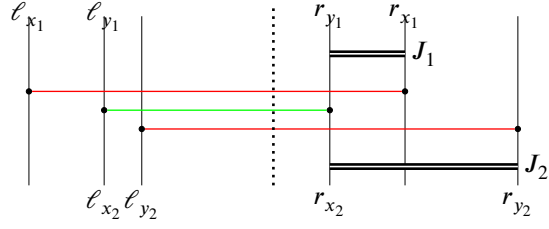

(ii) Assume  that  $J_1\not\subset J_2$. 
By Lemma \ref{slackends}, the left endpoint of a slack interval is the right endpoint of a green interval, and the right endpoint of a slack interval is the right endpoint of a red interval. Therefore, if endpoints of two slack intervals coincide, then the right endpoints of two distinct intervals of the representation would also coincide, which is not allowed in a representation. 
\end{proof}

Let $\mathcal{A}$ be the family of all TS- and ST-slack intervals in $\mathcal{R}$, and  
set
$\lambda_{\mathcal{R}}(J)$ for the slack value corresponding to $J\in\mathcal{A}$. 
Let the subfamily of maximal slack intervals be  
$${\mathcal{A}}_{max} = {\mathcal{A}}\setminus \{J\in {\mathcal{A}} : J\subsetneq H\ \text{\ for\ }\ H\in\mathcal{A}\}.$$ 
By Lemma \ref{lemma}(i), every ST-slack interval is in ${\mathcal{A}}_{max}$, furthermore, no  intervals in ${\mathcal{A}}_{max}$ contain each other. 
By a classical result due to Roberts \cite{Roberts}, there is a family of intervals $\mathcal{A}'_{max}$ with a bijection to $\mathcal{A}_{max}$, which preserves the order of the endpoints, and all intervals in $\mathcal{A}'_{max}$ are of the same length.
Consider the representation $\mathcal{R}^\prime$ of $\pi$ obtained in this way from $\mathcal{R}$, that is,  keep  left endpoints unchanged and relocate the right endpoints to the appropriate positions $r_1^\prime<\cdots < r_n^\prime$. Let ${\mathcal{A}}^\prime$ be the family of all slack intervals in 
${\mathcal{R}}^\prime$, and let $J^\prime$ be the slack interval corresponding to $J\in {\mathcal{A}}$. 
Then we  have
$\lambda_{\mathcal{R}^\prime}(J^\prime)=\lambda_{\mathcal{R}^\prime}(H^\prime)$ for every 
$J^\prime, H^\prime\in {\mathcal{A}}_{max}^\prime$.

By Lemma \ref{lemma}(ii), the endpoints of the slack intervals in  
${\mathcal{A}}_{max}^\prime$ are distinct. Now we modify ${\mathcal{R}}^\prime$ 
by replacing each ST-slack interval $[r_i^\prime,r_j^\prime]$ with 
$[r_i^*,r_j^*]$, where $r_i^*=r_i^\prime-\epsilon$ and $r_j^*=r_j^\prime+\epsilon$, and  $\epsilon>0$ is sufficiently small.
Denote $ {\mathcal{R}}^*$ the representation 
obtained in this way. 
Let ${\mathcal{A}}^*$ be the family of all slack intervals in 
${\mathcal{R}^*}$, let ${\mathcal{A}}^*_{max}\subset {\mathcal{A}}^*$ be the subfamily of all maximal slack intervals, and let $J^*\in {\mathcal{A}}^*$ be the slack interval corresponding to $J^\prime\in {\mathcal{A}}^\prime$. 
Then we have
 \begin{eqnarray}\label{adjustslacks}
 \lambda_{\mathcal{R}^*}(J^*) =\ \  \left\{\begin{array}{ccc}
\lambda_{\mathcal{R}^\prime}(J^\prime) +2\epsilon  & \text {\ if \ $J^*\in {\mathcal{A}}^*_{max}$ is an ST-slack interval} \\
\\
\lambda_{\mathcal{R}^\prime}(J^\prime) +\epsilon  &  \text {\ if \ $J^*\not\in {\mathcal{A}}^*_{max}$ is a TS-slack interval}  \\
\\
\lambda_{\mathcal{R}^\prime}(J^\prime)  & \text{\ if \ $J^*\in {\mathcal{A}}^*_{max}$ is a TS-slack interval} 
 \end{array}
  \right.
\end{eqnarray} 
Therefore, the representation  ${\mathcal {R}^*} $ of $\pi$ satisfies the the key inequality (\ref{keyprime}):
  \begin{equation*} 
{\rm max}_{\mathcal {R}^*} \text{TS}\leq  \lambda_{\mathcal{R}^\prime}(J^\prime) +\epsilon <
\lambda_{\mathcal{R}^\prime}(J^\prime) +2\epsilon =  {\rm min}_{\mathcal {R}^*}  \text{ST}.
\end{equation*}
This concludes the proof of the theorem.
 \end{proof}

\section{Representation of height-3 interval orders} 
\label{main}
The characterization of interval orders having a $k$-count representation and the complexity status of the corresponding recognition problems are still open for $k\geq 2$  in general. There are results for $k=2$ with additional restrictions \cite{BoyadzhiyskaDissertation,BGT,FMOS,JLORS}.
Here we discuss another particular case for $k=2$; in Theorem \ref{2count} we prove a characterization of height-3 posets having a 2-count interval representation.
 
While the exclusion of {$\mathbf{3}+\mathbf{1}$} subposets from interval orders completely characterizes semiorders (1-count interval orders), forbidding  {$\mathbf{4}+\mathbf{1}$} subposets from interval orders is not decisive regarding the existence of a 2-count interval representation. However, we conjecture that excluding springs and {$\mathbf{4}+\mathbf{1}$} produces $2$-count interval orders.
 
Here we prove a special case of Conjecture \ref{4plus1}, where the height
of $\Po$ is 3. In fact, in this case 
neither {$\mathbf{4}+\mathbf{1}$} nor  Springs  are subposets of $\Po$, since each has height 4. 

\begin{theorem}
\label{2count}
A height-3 interval order has a 2-count interval representation if and only if it has depth at most $2$. 
\end{theorem}
\begin{proof} We prove that the condition of $P$ having depth 2 is sufficient; necessity is obvious. The proof is a procedure that starts with the canonical representation of $P$, then this representation is modified to obtain one with
two different interval lengths in three main stages. 

Let $\mathcal{C}$ be  the canonical representation of $P$.
 In a diagram canonical intervals are positioned between vertical {\it lines} labeled $0,1,2,\ldots,m-1$, where $m$ is the magnitude
of $P$ representing the endpoints. In  $\mathcal{C}$ 
each line contains several coinciding endpoints. The first stage of the procedure 
consists of  `breaking ties' leading to another representation of $P$, where there are no coinciding interval endpoints of the same type.

It turns out that the bulk of this new representation of $P$ consists of intervals forming an interval representation of some permutation of depth-2. In the second stage of the procedure, we apply Theorem \ref{2countperm} leading to a 2-count representation of the `middle part' with interval lengths $\alpha<\beta$, where $\alpha$ is sufficiently large.
 
A canonical interval with an endpoint at $0$ or $m-1$  is called {\it extremal}.
Note that every line in a canonical representation contains a left endpoint and a right endpoint, thus $\mathcal{C}$ contains  $[0,0]$, $[m-1,m-1]$  (see \cite{Greenough}[Corollary 2.5]). Because $\Po$ has no chain containing more than 3 elements, the non-extremal intervals are pairwise intersecting. Therefore, by Helly's theorem \cite{Helly}, non-extremal intervals contain a common line $\ell_0$. All those intervals containing $\ell_0$ form the {\it middle part}. Note that the middle part may contain extremal intervals (see an example in Figure \ref{abra9}).

\begin{figure}[H]
\begin{center}
\begin{tikzpicture}[scale=.6]
\foreach \i in {6,...,28}{
    \draw[line width=.02em](.5*\i,1.5)--(.5*\i,11);
}
\draw[line width=.1em,dotted](8.75,1.5)--(8.75,11);
\node[]()at(8.75,11.15){$\ell_0$};
\node[]()at(3,11.15){$0$};
\node[]()at(14,11.15){$m-1$};

\node[A,label=left:\tiny$$](L10) at (3,10.5) {};
\node[A](R10) at (9,10.5) {};
\draw[line width=.02em,](L10)--(R10);

\node[A,label=left:\tiny$$](L11) at (3,10) {};
\node[A](R11) at (9.5,10) {};
\draw[line width=.02em,](L11)--(R11);

\node[A](L10) at (3,9.5) {};
\node[A,label=right:$$](R10) at (10.5,9.5) {};
\draw[line width=.02em](L10)--(R10);

\node[A](L20) at (4.5,7.5) {};
\node[A](R20) at (12,7.5) {};
\draw[line width=.02em](L20)--(R20);

\node[A](L4) at (5,7) {};
\node[A,label=right:$$](R4) at (10.5,7) {};
\draw[line width=.02em,](L4)--(R4);

\node[A](L5) at (5.5,6.5) {};
\node[A,label=right:$$](R5) at (10.5,6.5) {};
\draw[line width=.02em,](L5)--(R5);

\node[A,label=left:\tiny$$](L7) at (6,6) {};
\node[A](R7) at (11,6) {};
\draw[line width=.02em,](L7)--(R7);

\node[A,label=left:\tiny$$](L8) at (6.5,5.5) {};
\node[A](R8) at (11.5,5.5) {};
\draw[line width=.02em,](L8)--(R8);

\node[A,label=left:\tiny$$](L30) at (6.5,5) {};
\node[A](R30) at (12.5,5) {};%
\draw[line width=.02em](L30)--(R30);

\node[A,label=left:\tiny$$](L8) at (7,4.5) {};
\node[A,label=left:$$](R8) at (12,4.5) {};
\draw[line width=.02em,](L8)--(R8);

\node[A,label=left:\tiny$$](L2) at (3.5,9) {};
\node[A](R2) at (12,9) {};
\draw[line width=.02em,](L2)--(R2);

\node[A,label=left:\tiny$$](L3) at (4,8.5) {};
\node[A](R3) at (10,8.5) {};
\draw[line width=.02em,](L3)--(R3);

\node[A](L4) at (4,8) {};
\node[A,label=right:$$](R4) at (10.5,8) {};
\draw[line width=.02em,](L4)--(R4);

\node[A,label=left:\tiny$$](Lx) at (7.5,4) {};
\node[A,label=left:$$](Rx) at (12,4) {};
\draw[line width=.02em,](Lx)--(Rx);

\node[A,label=left:\tiny$$](L40) at (7.5,3.5) {};
\node[A,label=left:$$](R40) at (13.5,3.5) {};
\draw[line width=.02em](L40)--(R40);

\node[A,label=left:\tiny$$](Ly) at (8,3) {};
\node[A,label=left:$$](Ry) at (13,3) {};
\draw[line width=.02em,,](Ly)--(Ry);

\node[A,label=left:\tiny$$](Lz) at (8.5,2.5) {};
\node[A,label=left:$$](Rz) at (14,2.5) {};
\draw[line width=.02em,,](Lz)--(Rz);

\foreach \i in {0,...,5}{
    \draw[line width=.02em](3,8.5-.5*\i)--(3+.5*\i,8.5-.5*\i);}
\foreach \i in {0,...,2}{
    \draw[line width=.02em](3,5.5-\i)--(6+.5*\i,5.5-\i);}
     \draw[line width=.02em](3,3)--(7.5,3);
 \draw[line width=.02em](3,2.5)--(8,2.5);
  \draw[line width=.02em](3,2)--(8.5,2);
  
    \draw[line width=.02em](9,2)--(14,2);
    \draw[line width=.02em](13.5,3)--(14,3);   

\node[A]() at(3,9){};    \node[A]() at(14,3.5){};  

   \draw[line width=.02em](13,4)--(14,4);   
   \draw[line width=.02em](12.5,4.5)--(14,4.5);   
      \draw[line width=.02em](12,5.5)--(14,5.5);   
      \draw[line width=.02em](11.5,6.5)--(14,6.5);  
        \draw[line width=.02em](11,7)--(14,7);   
       \draw[line width=.02em](10.5,8.5)--(14,8.5);  
    \draw[line width=.02em](10,10)--(14,10);   
    \draw[line width=.02em](9.5,10.5)--(14,10.5);                         
\end{tikzpicture}
\end{center}
 \caption{A middle part of an interval order of depth-$2$}
\label{abra9}
\end{figure}
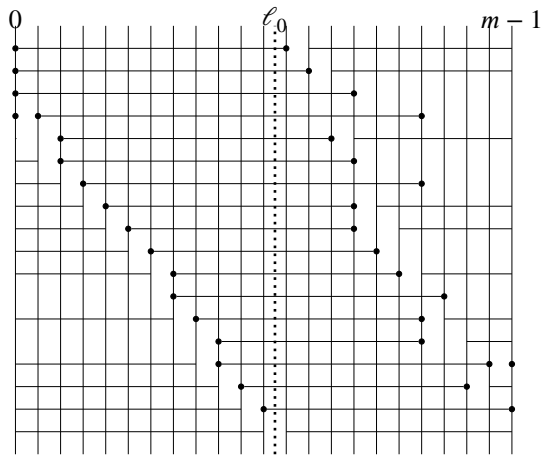
 The second stage produces a representation of $P$ that has depth-2 such that
intervals in the middle part have length $\alpha$ or $\beta$, and the length of further extremal intervals is less than $\alpha$.
  In the third stage the endpoints  of these `short' extremal intervals are moved  to obtain each length equal to $\alpha$.\\ 

\noindent\emph{1. Breaking ties.}
We may assume that $P$ has no twins\footnote{\ Elements of a poset having the same up-sets and the same down-sets are called twins.}. Due to this assumption, given two intervals,  either their right endpoints or their left endpoints are distinct.

 If there is a non-extremal singleton, then it is the only one at $\ell_0$. In this case, we replace the singleton with an  interval $I_0$ of length $\epsilon>0$ around $\ell_0$  in every  interval containing $\ell_0$. 
 Intervals with common endpoints are changed by shifting some endpoints into newly introduced lines holding the separated endpoints as follows. Left endpoints of the intervals at each line are shifted to left, right endpoints to right by eliminating inclusions.  
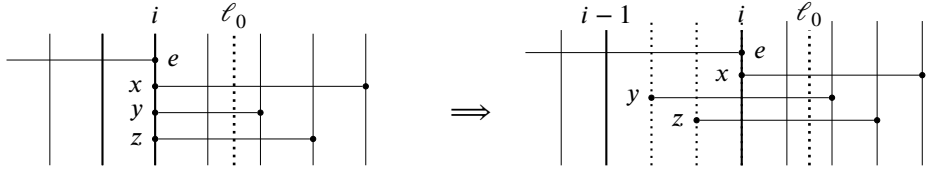
\begin{figure}[htp]
\begin{center}
\begin{tikzpicture}[scale=.7]
\foreach \i in {1,...,7}{
    \draw[line width=.02em](\i,1.5)--(\i,4);}
    
    \draw[line width=.1em,dotted](4.5,1.5)--(4.5,4);
        \node[]() at (4.5,4.35) {$\ell_0$};
        
   \draw[line width=.08em](2,1.5)--(2,4);
     \draw[line width=.08em](3,1.5)--(3,4);
         \node[]() at (3,4.35) {$i$};
         
       \node[](Lw) at (0,3.5) {};
\node[A,label=right:{$e$}](Rw) at (3,3.5) {};
\draw[line width=.02em](Lw)--(Rw);
    
\node[A,label=left:$x$](L2) at (3,3) {};
\node[A,label=left:$$](R2) at (7,3) {};
\draw[line width=.02em](L2)--(R2);

\node[A,label=left:$y$](L1) at (3,2.5) {};
\node[A,label=left:$$](R1) at (5,2.5) {};
\draw[line width=.02em](L1)--(R1);

\node[A,label=left:$z$](L2) at (3,2) {};
\node[A,label=left:$$](R2) at (6,2) {};
\draw[line width=.02em,,](L2)--(R2);

\node[]()at (9,2.5){$\Longrightarrow$};
\end{tikzpicture}
\begin{tikzpicture}[scale=.6]
\foreach \i in {4,...,7}{
    \draw[line width=.02em](\i,1)--(\i,4.2);}
    
       \draw[line width=.1em,dotted](4.5,1)--(4.5,4);
        \node[]() at (4.5,4.35) {$\ell_0$};
        
    \foreach \i in {1,3,2}{
    \draw[line width=.08em,dotted](\i,1)--(\i,4.2);}
      \draw[line width=.02em](-1,1)--(-1,4);
      \node[](Rw) at (3,4.35) {$ i$}; \node[](Rw) at (0,4.35) {$ i-1$};
      
   \draw[line width=.08em](0,1)--(0,4);
     \draw[line width=.08em](3,1)--(3,4);
     
     \node[](Lw) at (-2,3.5) {};
\node[A,label=right:{$e$}](Rw) at (3,3.5) {};
\draw[line width=.02em](Lw)--(Rw);

\node[A,label=left:$x$](L2) at (3,3) {};
\node[A,label=left:$$](R2) at (7,3) {};
\draw[line width=.02em](L2)--(R2);

\node[A,label=left:$y$](L1) at (1,2.5) {};
\node[A,label=left:$$](R1) at (5,2.5) {};
\draw[line width=.02em](L1)--(R1);

\node[A,label=left:$z$](L2) at (2,2) {};
\node[A,label=left:$$](R2) at (6,2) {};
\draw[line width=.02em,,](L2)--(R2);
\end{tikzpicture}
\end{center}
\caption{Breaking ties for the left endpoints at $i$}
\label{abra12}
\end{figure}

For each  $i$, $0< i<\ell_0$, there is a  set $E$ of intervals with left endpoint at  line $i$. We  distinguish the left endpoints of all these intervals  in the same order as their right endpoints (see Figure \ref{abra12}). Because non-extremal intervals contain $\ell_0$,  their right endpoints are greater than $\ell_0$. Therefore, the extremal interval  $I_e=[0,i]$ belongs to $\mathcal C$. 
We keep the right endpoint of $I_e$ at line $i$ together with the left endpoint of the largest interval in $E$,  all other left endpoints at $i$ move left into distinct new lines subdividing the gap $[i-1,i]$.
The  intervals modified in this way still represent  $\Po$.  
In the example of the figure above 
tie breaking  is applied with $E=\{x,y,z\}$ leading to \ $i-1< \ell_y < \ell_z < \ell_x=r_{e}=i$.
By switching to the dual of $\Po$, the same shifting is applied for the coinciding right endpoints in each line $j$  ($m-1>j>\ell_0$).  \\

\noindent\emph{2. Middle part.} Let $I_j=[\ell_j,r_j]_{j=1}^k$, be the  intervals (extremal, non-extremal), containing $\ell_0$. We may assume 
$$\ell_{\pi(1)}<\ell_{\pi(2)}<\cdots<\ell_{\pi(k)}  < \ell_0 < 
r_1<r_2<\cdots<r_k,$$ 
where $\pi$ is the permutation of $\{1,2,\ldots,k\}$, 
defined by the ordering of the left endpoints in the representation obtained in stage 2.

Due to the tie breaking rule applied in stage 1, and because  $\Po$ has depth 2,  there is no nested chain of more than two  intervals;  in particular, 
 $\pi$ is a depth-2 permutation. Applying 
Theorem \ref{2countperm}, we obtain an interval representation $\mathcal R$ of $\pi$ using interval lengths $\alpha < \beta$. We select the unit in the 2-count representation large enough so that $\alpha$ is larger than all interval lengths of extremal intervals not in the middle part.
 \\

\noindent\emph{3. Merging the  middle part with extremal intervals.}
Let O and M be the first and last line of the 2-count representation $\mathcal R$ obtained for the intervals of $P$ in the middle part. Now we include all extremal intervals not present in $\mathcal R$ (see the example below).

Because there is no containment between extremal intervals, after breaking the ties, each of these extremal intervals can be adjusted to have their length equal to $\alpha$ by extending them to left, whenever they used to contain line $0$ in $\mathcal C$, and  by extending to right, whenever they used to contain line $m-1$ in $\mathcal C$. 
\end{proof}

\begin{figure}[htp]
\begin{center}
\begin{tikzpicture}[scale=.45]

\node[X,label=right:{$$}](p1) at (-1,0) {6};
\node[X,label=left:{$$},label=right:{$$}](p3) at (2,0) {7};
\node[X,label=left:{$$},label=right:{$$} ](p7) at (4,0) {3};
\node[X,label=left:{$$},label=right:{$$} ](p9) at (6.5,0) {8};
\node[X,label=left:{$$},label=right:{$$} ](p11) at (8,0) {1};

\node[X](p4) at (-2,2.5) {5};
\node[X,label=left:{},label=right:{$$}](p6) at (-.5,2.5) {4};

\node[X](p10) at (1.5,4) {2};

\node[X,label=left:{$$},label=right:{$$} ](p12) at (5,5) {9};
\node[X](p8) at (7.5,9) {10};
\node[X](p5) at (-.5,10.5) {11};
\node[X,label=left:{$$},label=right:{$$} ](p2) at (4.65,12) {12};

\draw[line width=.07em](p4)--(p1)--(p6)(p3)--(p10)--(p1) (p4)--(p2)--(p6)(p10)--(p2)--(p7)(p9)--(p2)--(p11);
\draw[line width=.07em](p3)--(p12)--(p1)(p12)--(p9)(p7)--(p5);(p12)--(p3);
\draw[line width=.07em](p10)--(p5)--(p6) (p11)-- (p5)--(p9)(p10)-- (p8)--(p9)(p11)-- (p8);
\end{tikzpicture}\hskip1cm
\begin{tikzpicture}[scale=.8]
\foreach \i in {0,...,6}{
    \draw[line width=.02em](\i,-0.5)--(\i,4);
  \node[]() at (\i,4.27) {\tiny \i}; }

       \node[A,label=left:{  $6$}](LR1) at (0,3.5) {};
  
\node[A,label=left: $7$](L2) at (0,3) {};
\node[A](R2) at (1,3) {};
\draw[line width=.02em](L2)--(R2);

\node[A,label=left: $8$](L1) at (0,2.5) {};
\node[A,label=left:$$](R1) at (2,2.5) {};
\draw[line width=.02em](L1)--(R1);

\node[A,label=left: $3$](L2) at (0,2) {};
\node[A,label=left:$$](R2) at (4,2) {};
\draw[line width=.1em](L2)--(R2);

\node[A,label=left: $2$](L1) at (2,1.5) {};
\node[A,label=left:$$](R1) at (3,1.5) {};
\draw[line width=.1em](L1)--(R1);

\node[A,label=left: $4$](L2) at (1,1) {};
\node[A,label=left:$$](R2) at (4,1) {};
\draw[line width=.1em](L2)--(R2);

\node[A,label=left: $5$](L1) at (1,.5) {};
\node[A,label=left:$$](R1) at (5,.5) {};
\draw[line width=.1em](L1)--(R1);

      \node[A,label=right:{  $12$}](LR1) at (6,2) {};
    \node[A](L2) at (5,2.5) {};
\node[A,label=right: $11$](R2) at (6,2.5) {};
\draw[line width=.02em](L2)--(R2);
\node[A](L2) at (3,3.5) {};
\node[A,label=right: $9$](R2) at (6,3.5) {};
\draw[line width=.02em](L2)--(R2);

\node[A](L2) at (4,3) {};
\node[A,label=right: $10$](R2) at (6,3) {};
\draw[line width=.02em](L2)--(R2);

\node[A,label=left: $1$](L2) at (0,0) {};
\node[A,label=left:$$](R2) at (3,0) {};
\draw[line width=.1em](L2)--(R2);
\end{tikzpicture}
\end{center}
\caption{The Hasse diagram and canonical representation of the height-3, depth-2 interval order $P=(0,1,0,1,2,1,0,2,0,2,0,3)$ }
\label{Hasse}
\end{figure}
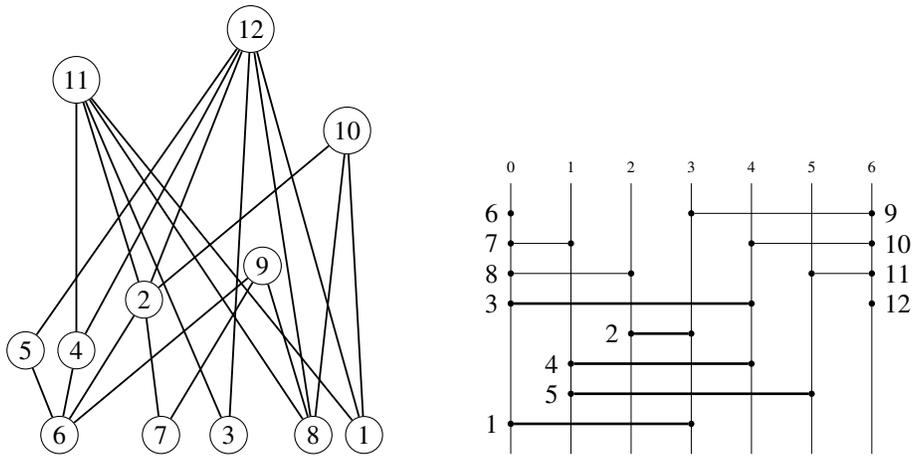

We conclude the section with an example. Consider the  depth-2, height-3 interval order $P=(0,1,0,1,2,1,0,2,0,2,0,3)$ identified here using its Ascent Sequence. (Notice that the ascent sequence labeling of the elements of $P$ is permuted for our purpose.) The Hasse diagram and canonical representation of $P$ is displayed in Figure \ref{Hasse};
the 2-count representation in Figure \ref{the2countrep} is obtained by the procedures in the proof of Theorems \ref{2count} and \ref{2countperm}.

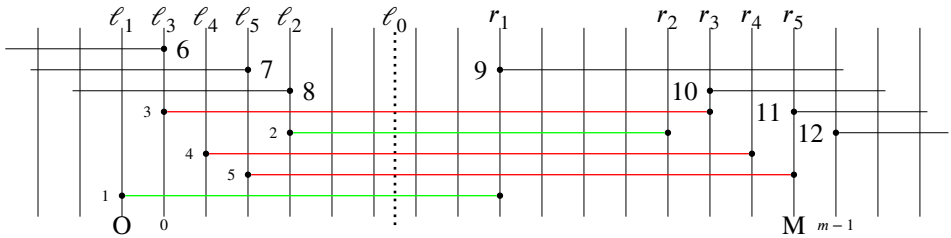
\begin{figure}[htp]
\begin{center}
\begin{tikzpicture}[scale=.5583]
\foreach \i in {-3,...,18}{
    \draw[line width=.02em](\i,-0.5)--(\i,4);}
 \node[]() at (0,-.7) {\tiny$0$};
   \node[]() at (16,-.7) {\tiny$m-1$};
      \node[]() at (-1,4.2) { $\ell_1$};   \node[]() at (0,4.2) { $\ell_3$};
         \node[]() at (1,4.2) { $\ell_4$};   \node[]() at (2,4.2) { $\ell_5$};
            \node[]() at (3,4.2) { $\ell_2$};
                  \node[]() at (8,4.2) { $r_1$};   \node[]() at (13,4.2) { $r_3$};
         \node[]() at (14,4.2) { $r_4$};   \node[]() at (15,4.2) { $r_5$};
            \node[]() at (12,4.2) { $r_2$};
    \node[]() at (-1,-.7) { O};
   \node[]() at (15,-.7) { M};
   
              \draw[line width=.1em,dotted](5.5,-.7)--(5.5,4.);
     \node[]() at (5.5,4.2) {$\ell_0$};   

             \node[A,label=right:{  $6$}](R1) at (0,3.5) {};
 \node[](L1) at (-4,3.5) {};
\draw[line width=.02em](L1)--(R1);
            
      \node[A,label=left:{  $12$}](L1) at (16,1.5) {};
  \node[](R1) at (18.9,1.5) {};
\draw[line width=.02em](L1)--(R1);

      \node[A,label=left:{  $11$}](L1) at (15,2) {};
  \node[](R1) at (18.4,2) {};
\draw[line width=.02em](L1)--(R1);
    
\node[](L2) at (-3.4,3) {};
\node[A,label=right: $7$](R2) at (2,3) {};
\draw[line width=.02em](L2)--(R2);

\node[](L1) at (-2.4,2.5) {};
\node[A,label=right: $8$](R1) at (3,2.5) {};
\draw[line width=.02em](L1)--(R1);
       
\node[A,label=left: \tiny$3$](L2) at (0,2) {};
\node[A,label=left:$$](R2) at (13,2) {};
\draw[line width=.05em,red](L2)--(R2);

\node[A,label=left:\tiny $2$](L1) at (3,1.5) {};
\node[A,label=left:$$](R1) at (12,1.5) {};
\draw[line width=.05em,green](L1)--(R1);

\node[A,label=left:\tiny $4$](L2) at (1,1) {};
\node[A,label=left:$$](R2) at (14,1) {};
\draw[line width=.05em,red](L2)--(R2);

\node[A,label=left: \tiny$5$](L1) at (2,.5) {};
\node[A,label=left:$$](R1) at (15,.5) {};
\draw[line width=.05em,red](L1)--(R1);

\node[A,label=left: \tiny$1$](L1) at (-1,0) {};
\node[A,label=left:$$](R1) at (8,0) {};
\draw[line width=.05em,green](L1)--(R1);

\node[A,label=left: $9$](L2) at (8,3) {};
\node[](R2) at (16.4,3) {};
\draw[line width=.02em](L2)--(R2);

\node[A,label=left: $10$](L2) at (13,2.5) {};
\node[](R2) at (17.4,2.5) {};
\draw[line width=.02em](L2)--(R2);
\end{tikzpicture}
\end{center}
\caption{A 2-count representation of $P=(0,1,0,1,2,1,0,2,0,2,0,3)$ with $\alpha=9$, $\beta=13$}
\label{the2countrep}
\end{figure}

The procedure in the proof of Theorem \ref{2count} starts with
the canonical representation of $P$ and 
 yields  the middle part permutation $\pi=\{1,3,4,5,2\}$ after breaking ties. 
  Then using the sorted 2-coloring $S=\{1,2\}$, $T=\{3,4,5\}$ Theorem \ref{2countperm} finds a solution with  $\alpha=9, \beta= 13$ for the middle part.  Theorem \ref{2count} completes the 2-count representation of $P$ by merging the middle part with the extremal intervals $6,7,8$ and $9,10,11,12$.

\section{Depth-3 permutation with no  $3$-count representation}
\label{example}
In the present section $\pi$ is a depth-3 permutation of $[n]$; let $\{I_j\}_{j\in[n]}$  be an interval representation of $\pi$. 
Recall that a sorted $3$-coloring $[T_1, T_2, T_3]$  of $\pi$ satisfies that 
 $T_i$ contains no inversion, for $i=1,2,3$, and 
 if $x\in T_i$, $y\in T_j$ and $I_x\subsetneq I_y$, then $i<j$.
A sorted $3$-coloring $[T_1, T_2, T_3]$ is {\it feasible}, provided  
there exist integers
$\alpha_1<\alpha_2<\alpha_3$, and an integral representation $[\ell_x,r_x]_{x=1}^n$ such that
 $r_x-\ell_x=\alpha_i$ for every $x\in T_i$.
   
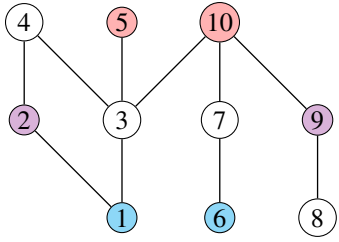
\begin{figure}[htp]   
    \begin{center}
\begin{tikzpicture}[scale=.65]

\node[Xg](1) at (2, 3) {$1$};
\node[Xg](6) at (4,3) {$6$};
\node[X](8) at (6,3) {$8$};

\node[Xb](2) at (0,5) {$2$};
\node[X](3) at (2,5) {$3$};
\node[X](7) at (4,5) {$7$};
\node[Xb](9) at (6,5) {$9$};

\node[X](4) at (0,7) {$4$};
\node[Xr](5) at (2,7) {$5$};
\node[Xr](10) at (4,7) {10};

\draw[line width=.05em](4)--(2)--(1)--(3)--(5)(10)--(3);
\draw[line width=.05em](6)--(7)--(10)--(9)--(8)(4)--(3);
\end{tikzpicture}
\end{center}
\caption{The inclusion order of the interval representation of $\pi=[4,2,5,10,3,1,7,6,9,8]$}
\label{abra7}
\end{figure}

There are several examples of depth-3 permutations with infeasible sorted 
3-colorings. This does not mean that the underlying permutation has no other feasible sorted $3$-coloring. However, any such permutation can be embedded into a larger depth-3 permutation such that this infeasible coloring is 
forced by all sorted 3-colorings of the larger permutation. 

 Computer search shows that the smallest depth-3 permutation with no 3-count representation is $\pi=[4,2,5,10,3,1,7,6,9,8]$ shown earlier in Figure \ref{abra6}. We prove next that this permutation has no 3-count representation.\\

Assume that $\pi$ has a  feasible 3-coloring.  The inclusion poset of $\{I_x\}_{x\in[10]}$ in Figure \ref{abra7} shows that 
$5,10\in T_3, 2,9\in T_2$ and $1,6\in T_1$ in every sorted $3$-coloring 
 $[T_1,T_2,T_3]$ of $\pi$. 
This means that when $\pi$ is restricted to the permutation $\pi^\prime=[2,5,10,1,6,9]$ on the subset $\{1,2,5,6,9,10\}$, 
the induced sorted $3$-coloring is
$T_1^\prime= \{1,6\}, \quad T_2^\prime=\{2,9\}, \quad 
T_3^\prime=\{5,10\}$. If this sorted $3$-coloring is feasible, then   there is a  representation  $[\ell_i,r_i]$, $i\in\{1,2,5,6,9,10\}$ of $\pi$ such that 
\begin{eqnarray}
\label{x}\ell_6-\ell_1&=&r_6-r_1\\
\label{y}\ell_9-\ell_2&=&r_9-r_2\\
\label{z}\ell_{10}-\ell_5&=&r_{10}-r_5
\end{eqnarray}

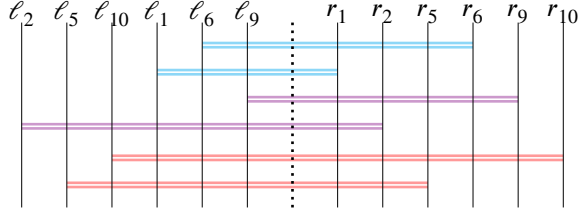
\begin{figure}[htp]
\begin{center}
\begin{tikzpicture}[scale=.6]
\draw[double,cyan!40,line width=.1em] (-4,-2.8) -- node [above] {} (0,-2.8); 
    \draw[cyan!40,double,line width=.1em] (-3,-2.2) -- node [above] {} (3,-2.2); 

     \draw[violet!40,double,line width=.1em] (-7,-4) -- node [above] {} (1,-4); 
\draw[double,line width=.1em,violet!40] (-2,-3.4) -- node [above] {} (4,-3.4); 
    
\draw[double,line width=.1em,red!40] (-6,-5.3) -- node [above] {} (2,-5.3);
    \draw[red!40,double,line width=.1em] (-5,-4.7) -- node [above] {} (5,-4.7); 
    
\foreach \i in {-7,...,-2}{
 \draw[line width=.02em](\i,-5.8)--(\i,-1.7);  }
 \node()at(-7,-1.5){$\ell_2$}; \node()at(-6,-1.5){$\ell_5$}; \node()at(-5,-1.5){$\ell_{10}$};
  \node()at(-4,-1.5){$\ell_1$}; \node()at(-3,-1.5){$\ell_6$}; \node()at(-2,-1.5){$\ell_9$};
 \foreach \i in {0,...,5}{
 \draw[line width=.02em](\i,-5.8)--(\i,-1.7);  }
  \node()at(0,-1.5){$r_1$}; \node()at(1,-1.5){$r_2$}; \node()at(2,-1.5){$r_5$};
  \node()at(3,-1.5){$r_6$}; \node()at(4,-1.5){$r_9$}; \node()at(5,-1.5){$r_{10}$};
  \draw[line width=.1em,dotted](-1,-5.8)--(-1,-1.7); 
\end{tikzpicture}
\end{center}
\caption{Infeasible sorted $3$-coloring}
\label{abra8}
\end{figure}
Observe that  $[\ell_5,\ell_{10}]\cap [\ell_1,\ell_6]=\varnothing$ and $[r_5,r_{10}]\cap [r_1,r_6]\neq\varnothing$; furthermore,\\
$[\ell_5,\ell_{10}]\cup [\ell_1,\ell_6]\subset [\ell_2,\ell_{9}]$ and $[r_2,r_{9}]\subset [r_5,r_{10}]\cup [r_1,r_6]$.
Then using (\ref{x}) and (\ref{z})  
\begin{eqnarray*}
\ell_9-\ell_2>(\ell_{10}-\ell_5) + (\ell_6-\ell_1)=
(r_{10}-r_5) + (r_6-r_1)>r_9-r_2 
\end{eqnarray*}
follows, contradicting (\ref{y}). Therefore,  $\pi^\prime$ and thus $\pi$ has no 3-count representation.
\section{Conclusions}
\label{conclusions}
In this section, we add a few comments and highlight significant open problems.

It is worth repeating that Graham's original $k$-count problem on interval orders remains open for $k \geq 2$.  In particular,
characterizing $2$-count interval orders remains unsolved despite having received attention for special cases 
\cite{BoyadzhiyskaDissertation,BGT,FMOS,JLORS}.
Adding another interesting special case, we conjecture that a $\{\text{spring, } \mathbf{4}+\mathbf{1}\}$-free interval order is $2$-count.

During his investigations of $2$-count interval orders, 
Fishburn proved \cite[Section 9, Theorem 1]{FishburnBook} that for arbitrary $k$, there are depth-$2$ interval orders with no $k$-count interval representation. 
His recursive construction includes an {$\mathbf{r}+\mathbf{1}$} subposet (an $r$-element chain with an additional incomparable element) for increasingly larger values of $r$.
For fixed $r$, does an upper bound on the count number of an {$\mathbf{r}+\mathbf{1}$}-free interval order follow from a bound on its depth?
For example, is a depth-$2$, {$\mathbf{4}+\mathbf{1}$}-free interval order necessarily $5$-count? 

We showed (Theorem \ref{2countperm}) that depth-$2$ permutations have a $2$-count representation. Meanwhile, we found several depth-$3$
permutations that are not $3$-count.  Can one describe all depth-$3$ permutations that are minimally not $3$-count?
Is there a bound on the count number of a permutation in terms of its depth? For example, does a depth-$3$ permutation necessarily have a $5$-count interval representation? 

\bibliographystyle{amsplain}
\bibliography{MutationReferences}

\providecommand{\bysame}{\leavevmode\hbox to3em{\hrulefill}\thinspace}
\providecommand{\MR}{\relax\ifhmode\unskip\space\fi MR }
\providecommand{\MRhref}[2]{%
  \href{http://www.ams.org/mathscinet-getitem?mr=#1}{#2}
}
\providecommand{\href}[2]{#2}
\begin{thebibliography}{10}

\bibitem{ascent2010}
Mireille Bousquet-M\'{e}lou, Anders Claesson, Mark Dukes, and Sergey Kitaev,
  \emph{{$(2+2)$}-free posets, ascent sequences and pattern avoiding
  permutations}, J. Combin. Theory Ser. A \textbf{117} (2010), no.~7, 884--909.
  \MR{2652101}

\bibitem{BoyadzhiyskaDissertation}
Simona Boyadzhiyska, \emph{Interval orders with restrictions on the interval
  lengths}, Wellesley College, 2016, Honors Thesis--Wellesley College.

\bibitem{BGT}
Simona Boyadzhiyska, Garth Isaak, and Ann~N. Trenk, \emph{Interval orders with
  two interval lengths}, Discrete Appl. Math. \textbf{267} (2019), 52--63.
  \MR{3996382}

\bibitem{FishburnBook}
Peter~C. Fishburn, \emph{Interval orders and interval graphs},
  Wiley-Interscience Series in Discrete Mathematics, John Wiley \& Sons, Ltd.,
  Chichester, 1985, A study of partially ordered sets, A Wiley-Interscience
  Publication. \MR{776781}

\bibitem{FMOS}
Mathew~C. Francis, L\'{\i}via~S. Medeiros, Fabiano~S. Oliveira, and Jayme~L.
  Szwarcfiter, \emph{On subclasses of interval count two and on {F}ishburn's
  conjecture}, Discrete Appl. Math. \textbf{323} (2022), 236--251. \MR{4503027}

\bibitem{Greenough}
Thomas~Lockman Greenough, \emph{Representation and enumeration of interval
  orders and semiorders}, ProQuest LLC, Ann Arbor, MI, 1976, Thesis
  (Ph.D.)--Dartmouth College. \MR{2626092}

\bibitem{Helly}
E.~Helly, \emph{{\"U}ber mengen konvexer k{\"o}rper mit gemeinschaftlichen
  punkten.}, Jber. Deutsch. Math. Verein. \textbf{32} (1923), 175--176.

\bibitem{JLORS}
Felix Joos, Christian L\"{o}wenstein, Fabiano de~S. Oliveira, Dieter
  Rautenbach, and Jayme~L. Szwarcfiter, \emph{Graphs of interval count two with
  a given partition}, Inform. Process. Lett. \textbf{114} (2014), no.~10,
  542--546. \MR{3219238}

\bibitem{Mirsky}
L.~Mirsky, \emph{A dual of {D}ilworth's decomposition theorem}, Amer. Math.
  Monthly \textbf{78} (1971), 876--877. \MR{288054}

\bibitem{Roberts}
Fred~S. Roberts, \emph{Indifference graphs}, Proof {T}echniques in {G}raph
  {T}heory ({P}roc. {S}econd {A}nn {A}rbor {G}raph {T}heory {C}onf., {A}nn
  {A}rbor, {M}ich., 1968), Academic Press, New York, 1969, pp.~139--146.
  \MR{0252267}

\bibitem{ScottSuppes}
Dana Scott and Patrick Suppes, \emph{Foundational aspects of theories of
  measurement}, J. Symbolic Logic \textbf{23} (1958), 113--128. \MR{115919}

\end{thebibliography}

\end{document}